\documentclass{amsart}

\usepackage[T1]{fontenc}
\usepackage{lmodern}
\usepackage{amsmath,amssymb,mathrsfs}
\usepackage[mathscr]{eucal}
\usepackage{enumitem}
\usepackage{microtype}
\usepackage[hidelinks,pdfusetitle]{hyperref}

\newtheorem{thm}{Theorem}[section]
\newtheorem{lem}[thm]{Lemma}
\newtheorem{prop}[thm]{Proposition}
\newtheorem{cor}[thm]{Corollary}
\numberwithin{equation}{section}

\theoremstyle{definition}
\newtheorem{defn}[thm]{Definition}
\newtheorem{ex}[thm]{Example}

\theoremstyle{remark}
\newtheorem{rem}[thm]{Remark}

\newcommand{\cl}{\operatorname{cl}}
\newcommand{\interior}{\operatorname{int}}
\newcommand{\Ls}{\operatorname{Ls}}

\newcommand{\cB}{\mathcal B}
\newcommand{\cC}{\mathcal C}
\newcommand{\cK}{\mathcal K}
\newcommand{\fD}{\mathfrak D}
\newcommand{\fA}{\mathcal{A}}
\newcommand{\N}{\mathbb N}

\begin{document}

\title[Bornologies and Kuratowski-Painlev\'e convergence]
{Bornologies and Kuratowski-Painlev\'e Convergence to the Empty Set Revisited}

\author{Basile Beyer de Ryke}
\address{12 Avenue Hamoir,
1180 Uccle, Belgium}
\email{basile.beyer@gmail.com}

\author{Gerald Beer}
\address{Department of Mathematics, California State University, Los Angeles,
5151 State University Drive, Los Angeles, California 90032, USA}
\email{gbeer@cslanet.calstatela.edu}

\date{September 5, 2026}

\hypersetup{
  pdftitle={Bornologies and Kuratowski-Painlev\'e Convergence to the Empty Set Revisited},
  pdfauthor={Basile Beyer de Ryke and Gerald Beer},
  pdfkeywords={bornology, metric bornology, Kuratowski--Painleve convergence,
    relatively compact set, Fell topology}
}

\begin{abstract}
Let $(X,\tau)$ be a noncompact metrizable space and let
$\langle T_n\rangle$ be a sequence of nonempty closed sets converging to the
empty set in the sense of Kuratowski--Painlev\'e.  The subsets of $X$ whose
closures meet at most finitely many terms form a bornology
$\cB_\tau(\langle T_n\rangle)$.  We give an internal criterion for inclusion
between two such bornologies and recover each of them from a canonical free
filter of closed sets.  In fact, a nontrivial bornology is sequence-induced
exactly when it has the enlargement property and its closed polar has
countable character.  Our principal result says that the bornology of
relatively compact sets is sequence-induced precisely when $X$ is locally
compact and separable; in the locally compact case the character of its
closed polar is the density of $X$.  In the penultimate section, we make a clever use of the sequence $T_n=\{n\}$ in
the real line to give a negative answer to a conjecture posed in 2001 about a
coarsest supertopology making the associated bornology a nontrivial metric
bornology. Finally, we end the paper with a more general result: a coarsest such metrizable supertopology exists
if and only if the original topology already makes the induced bornology a
nontrivial metric bornology, in which case the coarsest such supertopology is the
original topology itself.
\end{abstract}

\maketitle

\section{Introduction}

What should it mean for a sequence of nonempty closed subsets of a metrizable space to converge to
the empty set? A naive answer would be to equip the space with a compatible unbounded metric and
require that the sequence eventually leaves every bounded region. There are at least two problems with
this approach. First, the idea of bounded region is not well-defined, in that if the base metrizable space
is noncompact, there are uncountably many possible families of metrically bounded sets corresponding to
metrics compatible with the topology \cite[Theorem~15.5]{Beer2023}. Second, the sequence may seek to escape the space
through a hole in it. As an alternative, it has been consistently required that the sequence miss each compact
subset of the space eventually.

Such a sequence does not disappear without leaving a trace.  In
\cite{Beer2001}, Beer associated with a sequence $\langle T_n\rangle$
converging to the empty set the family of all subsets whose closures meet at most
finitely many of its terms.  This family is a nontrivial bornology, denoted here
by $\cB_\tau(\langle T_n\rangle)$.  In this way, a mode of escape from the
space produces a notion of boundedness on it.

Even very simple sequences can produce unexpectedly complicated
bornologies.  For example,
\[
 \{1\},\{2\},\{3\},\ldots
\]
in the real line induces a bornology with no countable base.  This example
will return in Section~6.  Before reaching it, we settle several basic
questions about the construction that were left unrecorded in
\cite{Beer2001}: how two induced bornologies can be compared, how finite
intersections are represented, and how metric bornologies arise from closed
sets escaping to infinity.

There is an unavoidable lower bound in all of these questions.  Every
relatively compact set belongs to $\cB_\tau(\langle T_n\rangle)$ for every
sequence converging to the empty set.  Thus the bornology of relatively compact subsets lies below every sequence-induced bornology.  This
brings up the principal question of the paper:
\[
 \text{When is this universal lower bound attained by a single sequence?}
\]
We prove that this happens precisely when $X$ is locally compact and
separable.

A second question is intrinsic.  The definition starts with a sequence, but
the resulting bornology need not remember that particular sequence.  How can
one recognize, from the bornology alone, whether some sequence converging to
the empty set induces it?  To answer this, we introduce its \emph{closed
polar}, a canonical filter of closed sets at infinity.  On a normal $T_2$
space, a nontrivial bornology is sequence-induced exactly when it has the
enlargement property and its closed polar has countable character.  A
representing sequence is then simply a choice of decreasing countable base
for the polar.

Finally, we return to the conjecture that closes \cite{Beer2001}.  Given a
sequence and a metrizable topology $\tau$, must there be a coarsest metrizable
\emph{supertopology} of $\tau$ for which the induced bornology becomes a nontrivial
metric bornology?  The answer is no.  Here and throughout the final two sections,
a \emph{supertopology} of $\tau$ means a topology $\sigma$ on $X$ with
$\tau\subseteq\sigma$.  A metrizable supertopology $\sigma$ of $\tau$ is called
\emph{successful} (for $\langle T_n\rangle$) when
$\cB_\sigma(\langle T_n\rangle)$ is a nontrivial metric bornology.  For the sequence
$\langle\{n\}\rangle$ in $\mathbb R$, we construct a decreasing family
of successful metrizable supertopologies whose intersection is the usual
topology, although that topology is not itself successful. We prove a stronger result, namely that a coarsest successful metrizable supertopology exists if and
only if the original topology already makes the induced bornology a nontrivial
metric bornology, in which case the coarsest supertopology is the original topology.

Section~2 supplies the required background.  Section~3 records the elementary
facts just mentioned.  Section~4 develops the closed polar and the intrinsic
representation theorem.  Section~5 treats the relatively compact bornology,
first directly and then through its polar.  Section~6 gives the
counterexample to the 2001 conjecture and Section~7 ends with the general theorem.
Throughout the manuscript, the following potential property of a bornology $\mathcal{B}$ will play a fundamental role:  whenever $A \in \mathcal{B}$ its closure is contained in the interior of another member of $\mathcal{B}$.  In fact we supply an unexpected characterization of this property to end Section 4.

\section{Preliminaries}

Throughout this article, we denote $\mathbb N=\{1,2,\ldots\}$. All topological spaces $(X,\tau)$ are assumed to satisfy the $T_2$ axiom.  If $A \subseteq X$, we denote its closure and interior by $\textrm{cl}(A)$ and $\textrm{int}(A)$.  Closures and interiors carry a
topological subscript when more than one topology is under discussion.

\begin{defn}[Relatively compact bornology]\label{def:relcompact-bornology}
A subset $A\subseteq X$ is called \emph{relatively compact} provided its closure is a compact subset of $X$.  We write
\[
 \cK_\tau(X)=\{A\subseteq X:\cl_\tau(A)\text{ is compact}\}
\]
for the bornology of relatively compact subsets of $(X,\tau)$.
\end{defn}
The $T_2$ axiom ensures that each compact subset is relatively compact; the $T_1$ axiom does not suffice in this regard. For a topological space $(X,\tau)$, the \emph{density} $d(X)$ is the least cardinality of a dense subset of $X$, and the \emph{weight} $w(X)$ is the least cardinality of a base for $\tau$.

Let $(X,\tau)$ be a topological space. A \emph{bornology} on $X$ is a family $\cB$ of subsets of
$X$ that covers $X$, is hereditary, and is closed under finite unions.  It
is \emph{nontrivial} if $X\notin\cB$.  A subfamily $\mathcal A$ is a
\emph{base} for $\cB$ if each member of $\cB$ is contained in some member
of $\mathcal A$. A base is called
\emph{closed, open}, or \emph{countable} when all its members are closed, all its members
are open, or it is countable, respectively.

Prominent examples of bornologies, with the relevant ambient structure
understood, include:
\begin{itemize}
\item the power set $\mathcal P(X)$;
\item the family of finite subsets of $X$;
\item the family of relatively compact subsets of a $T_1$ space \cite{BeerPajoohesh2026};
\item the family of metrically bounded subsets of a metric space;
\item the family of totally bounded subsets of a uniform space;
\item the family of Bourbaki bounded subsets of a uniform space
\cite{Beer2023,GarridoMerono2014,Hejcman1959,KunduAggarwal2023};
\item the family of
subsets of a topological vector space that are absorbed by each neighborhood of the origin.
\end{itemize}

For a metric $d$ on $X$, let
\[
 \cB_d(X)=\{A\subseteq X:A\text{ is }d\text{-bounded}\}.
\]
This is a bornology with a countable base.  A bornology on a metrizable
space $(X,\tau)$ is called a \emph{metric bornology} when it equals
$\cB_d(X)$ for some metric $d$ compatible with $\tau$.  If $d$ is bounded,
then $\cB_d(X)=\mathcal P(X)$; if $d$ is unbounded, then $\cB_d(X)$ is
nontrivial.  Thus a metric bornology is nontrivial precisely when every metric
representing it is unbounded.

The study of bornologies on metrizable spaces began in earnest with the
seminal paper of Hu \cite{Hu1949}.  We record the characterization that will
be used repeatedly below.

\begin{thm}[Hu \cite{Hu1949}]\label{thm:Hu}
Let $(X,\tau)$ be a metrizable space and let $\mathcal B$ be a bornology on
$X$.  Then $\mathcal B=\mathcal B_d(X)$ for some metric $d$ compatible with
$\tau$ if and only if both of the following conditions hold:
\begin{enumerate}
\item $\mathcal B$ has a countable base;
\item whenever $A\in\mathcal B$, there exists $B\in\mathcal B$ such that
      $A\subseteq B$ and $\cl(A)\subseteq\interior(B)$.
\end{enumerate}
\end{thm}
The second condition will be called the \emph{enlargement property}.  Upon a little reflection, it is clear that these two conditions amount to saying that $\mathcal{B}$ has a countable base, a closed base and an open base; no two of these conditions implies the third.  We record some interesting
standard examples.

\begin{ex}
In the rationals $\mathbb{Q}$ viewed as a subspace of $\mathbb{R}$ equipped with the usual topology, the bornology of finite subsets has a countable base, a closed base, but not an open base.  In fact, no nonempty member of the bornology is open.
\end{ex}

\begin{ex}
In an uncountable set with the discrete topology, the finite subsets have a
clopen base but no countable base.
\end{ex}

\begin{ex}\label{ex:integer-bornology}
In $\mathbb{R}$ equipped with the usual topology, the bornology
\[
 \{A\subseteq\mathbb R:A\cap\N\text{ is finite}\}
\]
has a countable base, an open base, but no closed base.  In Section~6 we consider the
smaller bornology consisting of those $A$ for which
$\cl(A)\cap\N$ is finite.  It has a closed base but no countable base.
\end{ex}

For background on metric bornologies and their order structure, we refer the reader to the monograph \cite{Beer2023}, as well as the recent article \cite{BeerPajoohesh2026}.

Let $\cC(X)$ denote the family of closed subsets of $(X,\tau)$, including the
empty set, and let $\cC_0(X)=\cC(X)\setminus\{\emptyset \}$.  For a
sequence $\langle T_n\rangle$ in $\cC_0(X)$, its
\emph{Kuratowski-Painlevé upper limit} is
\[
 \Ls_\tau(T_n)
 =\bigcap_{n=1}^{\infty}
   \cl\left(\bigcup_{j\geq n}T_j\right).
\]
We say that $\langle T_n\rangle$ is \emph{convergent to the empty set} when
$\Ls_\tau(T_n)=\emptyset$.

\begin{defn}[Fell topology]\label{def:Fell-topology}
The \emph{Fell topology} on $\mathcal C(X)$ has as a subbase all sets of the form
\[
 \{A\in\mathcal C(X):A\cap V\neq\emptyset\},
\]
where $V$ runs over $\tau$, together with all sets of the form
\[
 \{A\in\mathcal C(X):A\cap K=\emptyset\},
\]
where $K$ runs over the compact subsets of $X$.
\end{defn}

For sequences of nonempty closed subsets of a metrizable space, the following conditions are equivalent:
\begin{enumerate}
\item $\langle T_n\rangle$ is convergent to the empty set;
\item every point has a neighborhood meeting only finitely many terms
      $T_n$;
\item every compact subset of $X$ misses $T_n$ eventually.
\end{enumerate}
The equivalence of these statements is a special case of \cite[Theorem~5.2.10]{Beer1993}, which asserts, in particular, that in a metrizable space classical Kuratowski--Painlev\'e convergence of a sequence of closed sets is equivalent to its convergence in the Fell topology.

The Fell topology on $\mathcal{C}(X)$ is remarkable for its compactness in complete generality. When the metrizable space is locally compact, then the Fell topology on $\mathcal{C}(X)$ is also Hausdorff.  When the metrizable space is in addition separable, then the Fell topology becomes metrizable as well.  Moreover,
$\cC_0(X)$ is open in $\cC(X)$, since it is the set of closed sets that
meet $X$.  Thus, when $X$ is locally compact and separable,
$\cC_0(X)$ is itself locally compact, separable, and metrizable.  For details, we recommend
\cite{Attouch1984,Beer1993,KleinThompson1984}.

A sequence of nonempty closed sets in a compact metrizable space cannot converge to the empty set; we clearly see this fact with condition $(3)$ when the subset is $X$ itself.  We will be tacitly assuming, unless explicitly stated otherwise,  that all of our metrizable spaces are noncompact.

For a sequence $\langle T_n\rangle$ that converges to the empty set, Beer \cite{Beer2001} defined
\begin{equation}
 \cB_\tau(\langle T_n\rangle)
 =\{A\subseteq X:\cl(A)\cap T_n=\emptyset\text{ eventually}\}.
 \label{eq:induced-bornology}
\end{equation}
 The family is always a nontrivial bornology (as $X$ hits each $T_n$) and
satisfies the enlargement property.  Thus it is a metric bornology exactly
when it has a countable base \cite[Theorem~4.1]{Beer2001}.

There is no reason for that countable base to exist, even when the terms of
the sequence and the underlying space are completely familiar.  For
$T_n=\{n\}$ in the usual real line, the bornology has no countable base. We believe that the most intriguing result of \cite[Theorem~4.2]{Beer2001} identifies precisely the spaces on
which such pathology never occurs:  $\mathcal{B}(\langle T_n \rangle)$ is always a metric bornology if and only if the set of accumulation points of $X$ is compact.

This condition belongs to a well-developed part of metric space theory.
Rainwater showed that compactness of the derived set is equivalent to the
existence of a compatible metric inducing the finest compatible uniformity
\cite[Theorem~2]{Rainwater1959}.  With respect to such a metric, every
continuous map into an arbitrary metric space is uniformly continuous.  These
are the UC-spaces, also called Atsuji spaces; see
\cite{Atsuji1958,Beer1993,Beer2023,KunduAggarwal2023,Rainwater1959}.

\section{Convergence to the empty set revisited}

There were five elementary points about the construction in
\eqref{eq:induced-bornology} that were left unrecorded in \cite{Beer2001}.
None is difficult in isolation, but together they reveal much of the order
and representation structure of the induced bornologies.  We collect them
here because they will be used repeatedly.

\begin{prop}\label{prop:nontrivial}
Let $\langle T_n \rangle$ be a sequence of nonempty closed sets in a noncompact metrizable space $(X,\tau)$ convergent to the empty set.  Then $\forall n\in \N,\cl(\bigcup_{j=n}^{\infty}T_j)=\bigcup_{j=n}^{\infty}T_j.$
\end{prop}

\begin{proof}
Since $\langle T_n\rangle$ converges to the empty set, the family
$\{T_n:n\in\mathbb N\}$ is locally finite; that is, every point of $X$
has a neighborhood meeting only finitely many of the sets $T_n$.
Therefore, for each $n\in\mathbb N$, the family $\{T_j:j\geq n\}$ is
locally finite. Since a locally finite union of closed sets is closed, $\operatorname{cl}\left(\bigcup_{j=n}^{\infty}T_j\right)
=\bigcup_{j=n}^{\infty}T_j.$
\end{proof}

One would want to have internal necessary and sufficient conditions for the inclusion $\mathcal{B}(\langle T_n \rangle) \subseteq \mathcal{B}(\langle S_n \rangle)$ to hold. We find it informative to first give internal necessary and sufficient conditions for the inclusion to fail.

\begin{prop}\label{prop:failure-containment}
Let $\langle S_n \rangle$  and $\langle T_n \rangle$ be two sequences of nonempty closed sets in a noncompact metrizable space $(X,\tau)$ both convergent to the empty set.  Then $\mathcal{B}(\langle S_n \rangle)$ fails to contain $\mathcal{B}(\langle T_n \rangle)$ if and only if there exists a strictly increasing sequence of positive integers $\langle n_j \rangle$ and for all $j\in \mathbb N$, elements $s_{n_j}$ of $S_{n_j}$ for which $(\diamondsuit)$ holds:

$$(\diamondsuit) \ \ \exists k \in \mathbb{N}, \ \forall m \geq k, \ \{s_{n_j}: j \in \mathbb{N}\} \cap T_m = \emptyset.$$
\end{prop}

\begin{proof}
For sufficiency, the sequence $\langle s_{n_j} \rangle$ cannot have a cluster point, for if it did, the cluster point would have to lie in $\textrm{Ls}(S_n)$ which is impossible.  As a result, $S := \{s_{n_j}: j \in \mathbb{N}\}$ is a closed set, so if $(\diamondsuit)$ holds, $S$ fails to hit $T_m$ for all but finitely many  $m$. By definition, we have $S \in \mathcal{B}(\langle T_n \rangle)$, while $S \notin \mathcal{B}(\langle S_n \rangle)$ because $S$ hits infinitely many $S_n$.

Conversely, suppose $A \in \mathcal{B}(\langle T_n \rangle) \backslash \mathcal{B}(\langle S_n \rangle)$.  Without loss of generality, we may assume $A$ is closed.  By definition, there exists $k \in \mathbb{N}$  such that $\forall m \geq k, \ A \cap T_m = \emptyset$ while $A$ hits infinitely many $S_n$.  We can find $n_1 < n_2 < n_3 < \cdots$ and for each positive integer $j$ some point $s_{n_j} \in A$ belonging to $S_{n_j}.$  For each $m \geq k$, we have

$$\{s_{n_j}: j \in \mathbb{N}\} \cap T_m \subseteq A \cap T_m = \emptyset,$$ \vspace {2mm}

\noindent and condition $(\diamondsuit)$ is established.
\end{proof}

\begin{thm}\label{cor:containment}
Under the hypotheses of Proposition~\ref{prop:failure-containment},
\[
 \cB(\langle T_n\rangle)\subseteq\cB(\langle S_n\rangle)
\]
if and only if every selection $s_{n_j}\in S_{n_j}$, made along a strictly
increasing sequence of indices, meets $T_n$ for infinitely many $n$.
\end{thm}

The next proposition establishes that sequence-induced bornologies are stable under finite intersections, with a
particularly simple representation.

\begin{prop}\label{prop:unions}
Let $\langle S_n\rangle$ and $\langle T_n\rangle$ be two sequences of nonempty closed sets in a noncompact metrizable space $(X,\tau)$ convergent to the empty set.  Then $\langle S_n\cup T_n\rangle$ does so as well, and
\[
 \cB(\langle S_n\cup T_n\rangle)
 =\cB(\langle S_n\rangle)\cap\cB(\langle T_n\rangle).
\]
\end{prop}

\begin{proof}
Every compact set misses both $S_n$ and $T_n$ eventually, and therefore
misses $S_n\cup T_n$ eventually.  Moreover, for $A\subseteq X$, the set
$\cl(A)$ misses $S_n\cup T_n$ eventually if and only if it misses each of
$S_n$ and $T_n$ eventually.  This is precisely the asserted identity.
\end{proof}

Every nontrivial metric bornology can be represented by a decreasing sequence that converges to the empty set.

\begin{prop}\label{prop:metric-realization}
Let $d$ be an unbounded compatible metric on a noncompact metrizable space $(X,\tau)$. Then there exists a decreasing sequence
$\langle T_n \rangle$ in $\mathcal{C}_0(X)$ convergent to the empty set for which $$\mathcal{B}(\langle T_n \rangle) = \mathcal{B}_d(X).$$
\end{prop}

\begin{proof}
Fix $x_0\in X$ and set
\[
 T_n=\{x\in X:d(x,x_0)\geq n\}.
\]
Each $T_n$ is nonempty, because $d(\cdot,x_0)$ is unbounded.  The sets are closed, decrease to the empty set, and hence converge to the empty set.  Finally, $\cl(A)$ misses some $T_n$ if and only if
$A$ is $d$-bounded.
\end{proof}

We close the section with a less immediate observation: closed members of an
induced bornology can approach the whole space uniformly.  Here an incomplete
compatible metric is useful.

\begin{prop}\label{prop:uniformly-dense}
Let $(X,\tau)$ be a noncompact metrizable space.  There are a compatible metric
$d$, a convergent sequence $\langle T_n\rangle$ to the empty set, and an increasing sequence
of closed members $\langle B_n\rangle$ of
$\cB(\langle T_n\rangle)$ such that
\[
 \lim_{n\to\infty}\sup_{x\in X}d(x,B_n)=0.
\]
\end{prop}

\begin{proof}
A metrizable space is compact exactly when every compatible metric is
complete \cite[p.~347]{Engelking1977}.  Choose an incomplete compatible
metric $d$, let $(\widehat X,\widehat d)$ be its completion, and fix
$x_0\in\widehat X\setminus X$.  We can choose points $t_n\in X$ so that $r_n:=\widehat d(t_n,x_0)$ decreases strictly to $0$.

Set $T_n=\{t_n\}$ and
\[
 B_n=\{x\in X:\widehat d(x,x_0)\geq r_n\}.
\]
The sequence $\langle T_n\rangle$ has no cluster point in $X$, and hence it converges to the empty set.  Each $B_n$ is closed, the sequence $\langle B_n\rangle$ is
increasing, and $B_n\cap T_j=\emptyset$ for every $j>n$.  Thus
$B_n\in\cB(\langle T_n\rangle)$.  If
$\widehat d(x,x_0)<r_n$, then
\[
 d(x,t_n)\leq\widehat d(x,x_0)+\widehat d(x_0,t_n)<2r_n,
\]
whereas $x\in B_n$ otherwise.  Consequently
$\sup_{x\in X} d(x,B_n)\leq2r_n\to0$.
\end{proof}

These facts display both the flexibility and the nonuniqueness of the
construction.  Nontrivial metric bornologies always admit representations by decreasing
closed sets, while other induced bornologies may behave quite differently.
More importantly, the sequence itself is not canonical.  We next identify
the object at infinity that the bornology actually remembers.

\section{The closed polar and an intrinsic representation}

Throughout this section, $(X,\tau)$ is a normal $T_2$ space.

\begin{defn}[Closed polar]\label{def:closed-polar}
For a nontrivial bornology $\cB$ on $X$, its \emph{closed polar} is
\begin{equation}
 \fD_\tau(\cB)=
 \{F\in\cC(X): C\in\cB\text{ whenever }C\in\cC(X)
                      \text{ and }C\cap F=\emptyset\}.
 \label{eq:closed-polar}
\end{equation}
\end{defn}

The preceding results begin with a sequence and then retain only the
bornology it induces.  We now reverse this process.  If
\[
 F_n=\bigcup_{j=n}^{\infty}T_j
\]
is the \emph{closed tail tower} of a sequence that converges to the empty set, then a set belongs to the induced bornology exactly when
its closure avoids at least one of the $F_n$.  This suggests looking, for an
arbitrary bornology, at all closed sets that can play the role of such a
tail; Definition~\ref{def:closed-polar} records precisely this family.

Hu's original point of view associates with a nontrivial bornology $\cB$ its
cobounded filter
\[
 \mathscr I(\cB)=\{X\setminus B:B\in\cB\}.
\]
For the present problem, however, complements alone do not remember that the
definition tests $\cl(B)$ for membership.  We therefore work directly with closed
sets and retain exactly the information needed for that test.

We will use the following notation for the enlargement property:
\begin{equation}
 (\mathrm E_\tau)\qquad
 A\in\cB\Longrightarrow
 \text{there is }E\in\cB\text{ with }
 \cl(A)\subseteq\interior(E).
 \label{eq:enlargement}
\end{equation}

We notice that the closed polar is not, in general, the usual cobounded filter. Next, by a \emph{closed
filter} we mean a nontrivial filter in $\cC(X)$: it is upwardly stable among closed sets and stable under finite intersections.  Such a
filter $\fD$ is \emph{free} if $\bigcap\fD=\emptyset$. A subfamily $\fA \subseteq \fD$ is a \emph{base} for $\fD$ if every member of $\fD$ contains some member of $\fA$. Its \emph{character}
$\chi(\fD)$ is the least cardinality of a base.

Conversely, if $\fD$ is a closed filter, we introduce
\begin{equation}
 \cB_\tau[\fD]
 =\{A\subseteq X:\cl(A)\cap F=\emptyset
                   \text{ for some }F\in\fD\}.
 \label{eq:bornology-filter}
\end{equation}
We see that if $\fD$ is free, then $\cB_\tau[\fD]$ is a nontrivial bornology.
Indeed, the $T_2$ axiom and freeness put every singleton in this family,
while the hereditary property is immediate and closure under finite unions follows from the
filter property.  Nontriviality follows because no member of a nontrivial closed
filter is empty.

\begin{lem}[Closed polar lemma]\label{lem:closed-polar}
Let $\cB$ be a nontrivial bornology on $(X,\tau)$. Then
$\fD_\tau(\cB)$ is a nontrivial closed filter. If $\cB$ satisfies
$(\mathrm E_\tau)$, then $\fD_\tau(\cB)$ is free and
\[
\cB=\cB_\tau[\fD_\tau(\cB)]. \tag{$\spadesuit$}\label{eq:polar-recovery}
\]
Conversely, $(\spadesuit)$ implies $(\mathrm E_\tau)$.
\end{lem}

\begin{proof}
The set $X$ belongs to $\fD_\tau(\cB)$, whereas $\emptyset$ does not,
since $X\notin\cB$, and upward closure is immediate. Let
$F,G\in\fD_\tau(\cB)$, and suppose that $C$ is closed and
$C\cap F\cap G=\emptyset$. In the normal closed subspace $C$, the
disjoint closed sets $C\cap F$ and $C\cap G$ have disjoint relatively
open neighborhoods $U$ and $V$. The sets
$C_F=C\setminus U$ and $C_G=C\setminus V$ are closed in $X$, cover $C$,
and satisfy $C_F\cap F=C_G\cap G=\emptyset$. Thus
$C_F,C_G\in\cB$, whence $C\in\cB$. Therefore
$F\cap G\in\fD_\tau(\cB)$, proving that $\fD_\tau(\cB)$ is a
nontrivial closed filter.

Assume now that $(\mathrm E_\tau)$ holds. For each $x\in X$, choose
$E\in\cB$ such that $x\in\interior(E)$. Then
$X\setminus\interior(E)$ belongs to $\fD_\tau(\cB)$ and omits $x$, so
the filter is free. If $A\in\cB$, choose $E\in\cB$ such that
$\cl(A)\subseteq\interior(E)$. The set
$F=X\setminus\interior(E)$ lies in $\fD_\tau(\cB)$ and is disjoint from
$\cl(A)$, proving one inclusion in $(\spadesuit)$. Conversely, if
$\cl(A)$ is disjoint from some $F\in\fD_\tau(\cB)$, then the definition
of the closed polar gives $\cl(A)\in\cB$, and hence $A\in\cB$ by
heredity. Thus $(\spadesuit)$ holds.

Finally, suppose that $(\spadesuit)$ holds and let $A\in\cB$. There is
$F\in\fD_\tau(\cB)$ disjoint from $\cl(A)$. By normality, there exists
an open set $W$ such that
\[
\cl(A)\subseteq W\subseteq\cl(W)\subseteq X\setminus F.
\]
Since $\cl(W)$ is closed and disjoint from $F$, it belongs to $\cB$.
Therefore $\cl(A)\subseteq\interior(\cl(W))$, so $\cl(W)$ witnesses
$(\mathrm E_\tau)$.
\end{proof}

\begin{rem}\label{rem:balanced-filter}
The closed polar has a precise antecedent in filter theory.  Let
\[
 \mathscr F_\tau(\cB)
 =\{F\in\cC(X):X\setminus F\in\cB\},
\]
the closed trace of Hu's cobounded filter.  This is a closed filter.  Following
Carlson \cite[Theorem~3.6]{Carlson1989}, the \emph{balanced hull}
$b(\mathscr F)$ of a closed filter $\mathscr F$ is the family of all
closed sets $G$ such that every open set $U\supseteq G$ contains some
member of $\mathscr F$.  If $\cB$ satisfies $(\mathrm E_\tau)$, then
\begin{equation}
 \fD_\tau(\cB)=b(\mathscr F_\tau(\cB)).
 \label{eq:balanced-polar}
\end{equation}
Indeed, suppose first that $G\in\fD_\tau(\cB)$ and that $U$ is open with
$G\subseteq U$.  The closed set $C=X\setminus U$ belongs to $\cB$.
Choose $E\in\cB$ with $C\subseteq\interior(E)$.  Then
$F=X\setminus\interior(E)$ belongs to $\mathscr F_\tau(\cB)$ and is
contained in $U$.  Conversely, suppose that $G$ has Carlson's property
and let $C$ be closed and disjoint from $G$.  Some
$F\in\mathscr F_\tau(\cB)$ is contained in $X\setminus C$, whence
$C\subseteq X\setminus F\in\cB$.  Thus $G\in\fD_\tau(\cB)$.

Carlson's balanced hull therefore supplies the filter-theoretic envelope of
the construction.  The recovery of the bornology from this envelope, and
the countable-character criterion that converts it into sequences of closed sets converging to the empty set,
are the features used here.
\end{rem}

The closed polar lemma has removed the representing sequence from the
problem.  A sequence can be recovered precisely when this canonical filter
has countable character.

\begin{thm}[Intrinsic sequence criterion]\label{thm:intrinsic}
Let $\cB$ be a nontrivial bornology on a normal $T_2$ space $(X,\tau)$.  The
following conditions are equivalent.
\begin{enumerate}
\item There is a sequence $\langle T_n\rangle$ of nonempty closed sets converging to the empty set
      such that $\cB=\cB_\tau(\langle T_n\rangle)$.
\item The bornology $\cB$ satisfies $(\mathrm E_\tau)$, and its closed
      polar has a countable base.
\end{enumerate}
Moreover, under these conditions every decreasing countable base
$\langle F_n\rangle$ of $\fD_\tau(\cB)$
decreases to the empty set, and satisfies
\begin{equation}
 \cB=\{A\subseteq X:\cl(A)\cap F_n=\emptyset
                       \text{ for some }n\}.
 \label{eq:decreasing-base}
\end{equation}
\end{thm}

\begin{proof}
Assume first that $\cB=\cB_\tau(\langle T_n\rangle)$, and set
\begin{equation}
 F_n=\bigcup_{j=n}^{\infty}T_j.
 \label{eq:tail-sets}
\end{equation}
Although Proposition~\ref{prop:nontrivial} was stated for metrizable spaces, the same argument remains valid in the context of normal $T_2$ spaces. The convergence of $\langle T_n\rangle$ to the empty set implies that the family $\{T_n : n \in \mathbb{N}\}$ is locally finite. Hence each $F_n$ is closed, and $\langle F_n\rangle$ decreases to the empty set. We also have

\begin{equation}
 \cB=\{A\subseteq X:\cl(A)\cap F_n=\emptyset
                      \text{ for some }n\}.
 \label{eq:tail-representation}
\end{equation}
If $A\in\cB$, choose $n$ so that $\cl(A)\cap F_n=\emptyset$.
Normality gives an open set $W$ such that
\[
 \cl(A)\subseteq W\subseteq\cl(W)\subseteq X\setminus F_n.
\]
Then $\cl(W)\in\cB$ and $\cl(A)\subseteq\interior(\cl(W))$, proving
$(\mathrm E_\tau)$.

We claim that
\begin{equation}
 \fD_\tau(\cB)=
 \{H\in\cC(X):F_n\subseteq H\text{ for some }n\}.
 \label{eq:exact-polar}
\end{equation}
Every $F_n$ belongs to the closed polar, which proves one inclusion.  For
the other, take $H\in\fD_\tau(\cB)$ and suppose that
$F_n\nsubseteq H$ for every $n$.  Choose $x_n\in F_n\setminus H$ and
put $C=\{x_n:n\in\N\}$.  We verify that $C$ is closed.  If
$p\in\cl(C)\setminus C$, then the $T_2$ axiom allows us to remove
$x_1,\ldots,x_{k-1}$ from any neighborhood of $p$, and hence
\[
 p\in\cl(\{x_n:n\geq k\})\subseteq F_k
\]
for every $k$.  This contradicts $\bigcap_kF_k=\emptyset$.  Thus
$C$ is closed.  Since $C\subseteq X\setminus H$ and $H$ is in the
polar, $C\in\cB$.  But $C\cap F_n\neq\emptyset$ for every $n$,
contradicting \eqref{eq:tail-representation}.  This proves
\eqref{eq:exact-polar}; in particular, the polar has the countable base
$\langle F_n\rangle$.

Conversely, assume $(2)$.  By Lemma~\ref{lem:closed-polar}, the polar is free
and \eqref{eq:polar-recovery} holds.  If
$\{H_n:n\in\N\}$ is a countable base, put
\[
 F_n=H_1\cap\cdots\cap H_n.
\]
Then $\langle F_n\rangle$ is a decreasing base of nonempty closed sets and
$\bigcap_nF_n=\emptyset$.  It therefore converges to the empty set, and
\eqref{eq:polar-recovery} gives
$\cB=\cB_\tau(\langle F_n\rangle)$.  The final assertion follows from the
same argument.
\end{proof}

The preceding theorem naturally yields an order-theoretic correspondence.

\begin{thm}
    \label{thm:order}
The assignments
\[
 \cB\longmapsto\fD_\tau(\cB),
 \qquad
 \fD\longmapsto\cB_\tau[\fD]
\]
are mutually inverse, inclusion-preserving bijections between the nontrivial
bornologies induced by sequences converging to the empty set and the free countably based closed
filters.
\end{thm}

\begin{proof}
Theorem~\ref{thm:intrinsic} and Lemma~\ref{lem:closed-polar} show that the
first composition is the identity.  For the second composition, let $\fD$ be a free countably based closed
filter and choose a decreasing base $\langle F_n\rangle$.  Then
$\cB_\tau[\fD]$ is induced by $\langle F_n\rangle$, and certainly
$\fD\subseteq\fD_\tau(\cB_\tau[\fD])$.  If a closed set $H$ does not
belong to $\fD$, then $F_n\nsubseteq H$ for every $n$.  The diagonal
construction in the proof of Theorem~\ref{thm:intrinsic} produces a closed
set $C\subseteq X\setminus H$ meeting every $F_n$.  Hence
$C\notin\cB_\tau[\fD]$, and $H$ is not in the polar.  The second
composition is therefore the identity.  Both assignments preserve inclusion.
\end{proof}

Notice that, despite the usual order-reversing flavor of the word ``polar,''
this correspondence preserves inclusion: here
$\cB_1\subseteq\cB_2$ implies
$\fD_\tau(\cB_1)\subseteq\fD_\tau(\cB_2)$.
Thus the representation is not merely existential.  The bornology remembers
a unique closed filter, while its possible representing sequences amount to
different choices of decreasing countable bases for that filter.

In the metrizable case, \cite[Theorem~4.1]{Beer2001} identifies the metric
subclass immediately.

\begin{cor}[Regular polar bases]\label{cor:regular-polar}
Let $(X,\tau)$ be metrizable and let $\cB$ be induced by a
sequence converging to the empty set.  Then $\cB$ is a metric bornology if and only if its closed polar
has a decreasing base $\langle F_n\rangle$ such that, for every $k$,
there is $n>k$ for which
\[
 F_n\subseteq\interior(F_k).
\]
\end{cor}

\begin{proof}
Suppose first that $\cB$ is a metric bornology, and choose a
sequence $\langle T_n\rangle$ converging to the empty set and inducing it.  Its closed tails
\[
 E_n=\bigcup_{j=n}^{\infty}T_j
\]
form a decreasing base of $\fD_\tau(\cB)$ by
Theorem~\ref{thm:intrinsic}.  Beer \cite[Theorem~4.1]{Beer2001} shows that
metricity is equivalent to the existence of a subsequence
$\langle E_{n_j}\rangle$ satisfying
\[
 E_{n_{j+1}}\subseteq\interior(E_{n_j})\qquad(j\in\N).
\]
This subsequence is still a base of the polar.

Conversely, suppose the polar has a decreasing base
$\langle F_n\rangle$ with the stated property.  Recursively choose
$n_1<n_2<\cdots$ so that
$F_{n_{j+1}}\subseteq\interior(F_{n_j})$.  This subsequence induces the
same bornology as the original base.  Another application of
\cite[Theorem~4.1]{Beer2001} shows that this bornology is determined by an
unbounded compatible metric.
\end{proof}
At this point, the role of countability is transparent; without countability, the same construction
produces a net.
The net formulation allows us to gain a new understanding of property $(E_\tau)$.

\begin{prop}\label{prop:net}
A nontrivial bornology $\cB$ satisfies $(\mathrm E_\tau)$ if and only if
there is a decreasing net $\langle F_i\rangle_{i\in I}$ of nonempty closed
sets such that
\[
 \bigcap_{i\in I}F_i=\emptyset,
 \qquad
 \cB=\{A\subseteq X:\cl(A)\cap F_i=\emptyset\text{ eventually}\}.
\]
\end{prop}

\begin{proof}
Assume first that $\cB$ satisfies $(\mathrm E_\tau)$.  Take
$I=\fD_\tau(\cB)$, with $F\preceq G$ when $G\subseteq F$, and use
the net whose $F$-th term is $F$.  This is a directed decreasing net.
Moreover, a closed set misses some $F\in\fD_\tau(\cB)$ if and only if it
misses all $G\succeq F$.  Lemma~\ref{lem:closed-polar} now says that the
net has empty intersection and gives the asserted representation.

Conversely, suppose that $\cB$ has such a representation and take
$A\in\cB$.  There is an index $i$ for which
$\cl(A)\cap F_i=\emptyset$.  By normality, choose an open set $W$ with
\[
 \cl(A)\subseteq W\subseteq\cl(W)\subseteq X\setminus F_i.
\]
Since the net is decreasing, $\cl(W)$ also belongs to $\cB$, and it
witnesses $(\mathrm E_\tau)$.
\end{proof}

For a decreasing net, the Kuratowski--Painlev\'e upper limit is
\[
 \bigcap_{i\in I}\cl\left(\bigcup_{j\succeq i}F_j\right)
 =\bigcap_{i\in I}F_i.
\]
Thus the empty-intersection condition in Proposition~\ref{prop:net} is
precisely Kuratowski--Painlev\'e convergence of the net to the empty set.

\section{The relatively compact bornology}

We return to the setting of noncompact metrizable spaces for a principal
result of the article. As usual, let $\langle T_n \rangle$  be a sequence of nonempty closed sets convergent to the empty set.   If $C$ is relatively compact, then the terms of $\langle T_n \rangle$ miss $\cl(C)$ eventually, and so by definition,  $C$ belongs to $\mathcal{B}(\langle T_n \rangle)$. Minimally, then, $\mathcal{B}(\langle T_n \rangle)$ contains the bornology of relatively compact subsets $\mathcal{K}_\tau(X)$.  This brings up a basic question:  in exactly which metrizable spaces does there exist a sequence $\langle T_n \rangle$ for which $\mathcal{K}_\tau(X) = \mathcal{B}(\langle T_n \rangle)$?

The intrinsic sequence criterion, Theorem~\ref{thm:intrinsic}, already suggests the answer.  We first give a direct
proof because it displays the two obstructions separately: the closed tails
force local compactness, while their shrinking neighborhoods force
$\sigma$-compactness and hence separability.  We need only the following
elementary lemma.

\begin{lem}\label{lem:shrinking-neighborhoods}
Let $d$ be a compatible metric for the noncompact metrizable space $(X,\tau)$, and let
$\langle E_n\rangle$ be a decreasing sequence of nonempty closed sets with
empty intersection.  Put
\[
 V_n=\{x\in X:d(x,E_n)<\frac1n\}.
\]
Then $\bigcap_nV_n=\emptyset$.
\end{lem}

\begin{proof}
If $v\in\bigcap_nV_n$, choose $e_n\in E_n$ with
$d(e_n,v)<\frac1n$.  Then the sequence $\langle e_n\rangle $ converges to $v$.  Each $E_n$ contains a tail of this
sequence and is closed, so $v\in E_n$ for every $n$. We obtain a contradiction, and the claim follows.
\end{proof}

We recall the following result of Vaughan \cite{Vaughan1937} from 90 years ago: a metrizable space has a compatible metric for which closed and bounded sets are compact if and only if its topology is both locally compact and separable.  For a more inclusive result, we refer the reader to \cite[Theorem~15.4]{Beer2023}.

\begin{thm}\label{lion}  Let $(X,\tau)$ be a noncompact metrizable space. Then the space admits a sequence of nonempty closed sets $\langle T_n \rangle$ convergent to the empty set for which $\mathcal{K}_\tau(X) = \mathcal{B}(\langle T_n \rangle)$ if and only if $(X,\tau)$ is both locally compact and separable.
\end{thm}

\begin {proof}  Necessity is the harder part.  Suppose $\langle T_n \rangle$ in $\mathcal{C}_0(X)$ converges to the empty set and $\mathcal{K}_\tau(X) = \mathcal{B}(\langle T_n \rangle)$.  Let $p \in X$ be arbitrary. Since $\textrm{Ls}(T_n) = \emptyset$, there exists $n \in \mathbb{N}$ such that $p \notin \cl(\bigcup_{j=n}^\infty T_j)$. Fix a metric $d$ compatible with $\tau$. We can choose $\delta > 0$ such that

$$\{x \in X: d(x,p) \leq \delta\} \cap \cl(\bigcup_{j=n}^\infty T_j) = \emptyset.$$ \vspace {2 mm}

\noindent This implies that $\forall j \geq n, \ \{x \in X: d(x,p) \leq \delta\}\cap T_j = \emptyset$ and since closed balls are closed sets, by definition, $\{x \in X: d(x,p) \leq \delta\} \in \mathcal{B}(\langle T_n \rangle)$.  By assumption, $\{x \in X: d(x,p) \leq \delta\}$ must be relatively compact, and so $p$ has a compact neighborhood. This gives local compactness for the space.

To show separability, it is enough to show $\sigma$-compactness, since a countable union of separable subspaces is separable.  For each $n$, set
\[
 E_n=\cl\left(\bigcup_{j=n}^{\infty}T_j\right)
 \quad\text{and}\quad
 V_n=\left\{x\in X:d(x,E_n)<\frac1n\right\}.
\]
By Lemma~\ref{lem:shrinking-neighborhoods}, $\bigcap_{n=1}^{\infty}V_n=\emptyset$.  Since $E_n\subseteq V_n$, the closed set $X\setminus V_n$ belongs to $\mathcal{B}(\langle T_n\rangle)=\mathcal{K}_\tau(X)$, and is therefore compact.  De Morgan's laws now give
\[
 X=\bigcup_{n=1}^{\infty}(X\setminus V_n),
\]
so $X$ is $\sigma$-compact and hence separable.

For sufficiency, assuming local compactness and separability, Vaughan's theorem allows us to equip $X$ with a compatible metric $d$ for which closed and bounded subsets are compact. Since $X$ is closed but not compact, the metric space $(X,d)$ is not bounded.  Fixing $x_0 \in X$, the decreasing sequence of nonempty closed sets $\langle T_n \rangle$ where $T_n := \{x \in X: d(x,x_0) \geq n\}$ clearly converges to the empty set because for each $n \in \mathbb{N}, \cl(\bigcup_{j=n}^\infty T_j) = T_n$.  Now if $A \in \mathcal{B}(\langle T_n \rangle)$, there exists $n$ such that
$\cl(A) \cap T_n = \emptyset$.  This means $\cl(A)$ is closed and bounded and so it is compact. By definition
$A$ must be relatively compact.  We have shown that $\mathcal{B}(\langle T_n \rangle) \subseteq \mathcal{K}_\tau(X)$.  The reverse inclusion always holds.  \end {proof} \vspace {4 mm}

The theorem asks when the universal lower bound is attained by one sequence.
Even when it is not attained, it remains exact after all sequence-induced
bornologies are intersected.

\begin{cor}\label{cor:least-induced}
For every noncompact metrizable space $(X,\tau)$,
\begin{equation}
 \cK_\tau(X)=
 \bigcap\{\cB_\tau(\langle T_n\rangle):
          \langle T_n\rangle\text{ converges to the empty set}\}.
 \label{eq:toad}
\end{equation}
Consequently, the sequence-induced bornologies have a least member if and
only if $X$ is locally compact and separable.  In that case the least member
is $\cK_\tau(X)$.
\end{cor}

\begin{proof}
Inclusion of the left-hand side in the right-hand side has already been observed.  If $A\notin\cK_\tau(X)$, the noncompact
metric space $\cl(A)$ contains a sequence $\langle y_n\rangle$ with no
convergent subsequence.  Fix a compatible metric $d$ and choose $a_n\in A$
with $d(a_n,y_n)<\frac1n$.  A convergent subsequence of $\langle a_n\rangle$
would give a convergent subsequence of $\langle y_n\rangle$.  Thus, after
passing to a subsequence, the points $a_n$ are distinct and have no cluster
point in $X$.  Then
$T_n=\{a_n\}$ converges to the empty set, while $\cl(A)$ meets every $T_n$.  Hence
$A\notin\cB_\tau(\langle T_n\rangle)$, proving
\eqref{eq:toad}.

Any least sequence-induced bornology would have to equal the intersection in
\eqref{eq:toad}.  The final statement now follows from
Theorem~\ref{lion}.
\end{proof}

We now return to the closed polar to see
why the two hypotheses in Theorem~\ref{lion} are inevitable and to obtain a
cardinal refinement.  We first identify the polar explicitly.

\begin{prop}

\label{prop:polar-compact}
For every closed set $F\subseteq X$,
\begin{equation}
 F\in\fD_\tau(\cK_\tau(X))
 \quad\Longleftrightarrow\quad
 X\setminus F\text{ is relatively compact}.
 \label{eq:compact-polar}
\end{equation}
\end{prop}

\begin{proof}
Put $U=X\setminus F$.  If $\cl(U)$ is compact, every subset of $U$
that is closed in $X$ is compact, so $F$ is in the polar.

Conversely, suppose that $U$ is not relatively compact.  Fix a compatible
metric $d$.  The noncompact metric space $\cl(U)$ contains a sequence
$\langle y_n\rangle$ with no convergent subsequence.  Choose
$x_n\in U$ with $d(x_n,y_n)<\frac1n$, passing to a subsequence so that the
points are distinct.  Any convergent subsequence of $\langle x_n\rangle$
would give a convergent subsequence of $\langle y_n\rangle$.  Consequently
$C=\{x_n:n\in\N\}$ is a closed noncompact
subset of $U$.  Thus $C\notin\cK_\tau(X)$, and $F$ is not in the
polar.
\end{proof}

In particular,
\begin{equation}
\fD_\tau(\cK_\tau(X))
 =\{X\setminus U:U\text{ is open and relatively compact}\}.
 \label{eq:cocompact-polar}
\end{equation}

The intersection of all members of this polar is the set of points having no relatively compact open neighborhood.

\begin{prop}\label{prop:compact-polar-character}
Let $(X,\tau)$ be a noncompact locally compact metrizable space. Then
\begin{equation}
 \chi(\fD_\tau(\cK_\tau(X)))=d(X)=w(X),
 \label{eq:character-density}
\end{equation}
where $d(X)$ and $w(X)$ denote the density and weight of $X$,
respectively.
\end{prop}

\begin{proof}
Let $\{F_\alpha:\alpha<\lambda\}$ be a base of the polar and
put $U_\alpha=X\setminus F_\alpha$.  The sets $U_\alpha$ cover $X$.
Indeed, if $V$ is a relatively compact open neighborhood of $x$, some
$F_\alpha$ is contained in $X\setminus V$, and then
$x\in V\subseteq U_\alpha$.  A finite polar base would make $X$ a finite
union of relatively compact sets and hence compact, so $\lambda$ is
infinite.  Each $\cl(U_\alpha)$ is compact and metrizable, hence separable.
Choose a countable dense subset $D_\alpha$ of each closure.  Since the
sets $U_\alpha$ cover $X$, the union of the $D_\alpha$'s is dense in
$X$, and its cardinality is at most $\lambda$.  Thus $d(X)\leq\lambda$.

For the reverse inequality, let $\kappa=d(X)=w(X)$.  From the cover of
$X$ by relatively compact open sets, choose a subcover
$\{V_\alpha:\alpha<\lambda\}$ with $\lambda\leq\kappa$; every open cover
has a subcover of cardinality at most the weight.  The finite
unions of the $V_\alpha$'s form a family of cardinality at most $\kappa$
that is cofinal among relatively compact open subsets: compactness of
$\cl(U)$ gives a finite covering whenever $U$ is relatively compact and
open.  Taking complements and using Proposition~\ref{prop:polar-compact} gives a
polar base of cardinality at most $\kappa$.  This proves
\eqref{eq:character-density}.
\end{proof}

Taking complements in \eqref{eq:cocompact-polar} shows that polar bases
correspond to cofinal families of relatively compact open sets.  Since
$\cK_\tau(X)$ has an open base when $X$ is locally compact,
\[
 \chi(\fD_\tau(\cK_\tau(X)))
 =\operatorname{cof}(\cK_\tau(X)),
\]
where $\operatorname{cof}(\cK_\tau(X))$ denotes the least cardinality of a base for the bornology $\cK_\tau(X)$.
For noncompact locally compact metrizable spaces, Eshed et al.\ proved that
the cardinal on the right equals $d(X)$
\cite[Theorem~4.8(1)]{EshedEtAl2020}.  Thus the numerical equality in
\eqref{eq:character-density} is already implicit in their result; the direct
closed-polar interpretation is the form needed to determine representability by a sequence that converges to the empty set.

Theorem~\ref{lion} now also follows from
Theorem~\ref{thm:intrinsic}, the standard equivalence between local
compactness of $X$ and the enlargement property for $\cK_\tau(X)$, and
Proposition~\ref{prop:compact-polar-character}: countable character of the
polar is equivalent to separability.

Theorem~\ref{lion} tells us when some sequence works.  In applications,
however, one is often given a particular sequence.  The same machinery says
exactly when that sequence realizes the relatively compact bornology.

\begin{thm}\label{thm:fixed-sequence}
Let $\langle T_n\rangle$ be a sequence of nonempty closed sets converging to the empty set, and put
\begin{equation}
 F_n=\bigcup_{j=n}^{\infty}T_j,
 \qquad
 U_n=X\setminus F_n.
 \label{eq:fixed-tails}
\end{equation}
Then
\begin{equation}
 \cB_\tau(\langle T_n\rangle)=\cK_\tau(X)
 \quad\Longleftrightarrow\quad
 U_n\text{ is relatively compact for every }n.
 \label{eq:fixed-criterion}
\end{equation}
When these conditions hold, $\langle U_n\rangle$ is an increasing cofinal
open cover of $X$ by relatively compact sets, and
$\langle F_n\rangle$ is a decreasing base of
$\fD_\tau(\cK_\tau(X))$.
\end{thm}

\begin{proof}
The tail sets $F_n$ are closed and decrease to the empty set; hence the
sets $U_n$ form an increasing open cover of $X$.  Also,
\begin{equation}
 A\in\cB_\tau(\langle T_n\rangle)
 \quad\Longleftrightarrow\quad
 \cl(A)\cap F_n=\emptyset\text{ for some }n.
 \label{eq:fixed-membership}
\end{equation}

Suppose first that the two bornologies are equal.  Every subset
$C\subseteq U_n$ that is closed in $X$ misses $F_n$, and therefore is
compact.  Proposition~\ref{prop:polar-compact} shows that $U_n$ is relatively
compact.

Conversely, suppose that every $U_n$ is relatively compact.  If
$A$ belongs to the induced bornology, then
$\cl(A)\subseteq U_n$ for some $n$, so $\cl(A)$ is compact.  If
$A\in\cK_\tau(X)$, the compact set $\cl(A)$ is covered by the increasing
family $\langle U_n\rangle$, and is therefore contained in one of its
members.  Equation \eqref{eq:fixed-membership} gives the reverse inclusion.

The same compactness argument proves cofinality of the exhaustion.  Finally,
if $F$ belongs to the polar, then $X\setminus F$ is relatively compact
and lies in some $U_N$.  Equivalently, $F_N\subseteq F$, which proves
that $\langle F_n\rangle$ is a polar base.
\end{proof}

\begin{cor}\label{cor:regular-exhaustion}
Under the equivalent conditions of Theorem~\ref{thm:fixed-sequence}, there
are strictly increasing indices $n_1<n_2<\cdots$ such that
\[
 \cl(U_{n_k})\subseteq U_{n_{k+1}},
 \qquad\text{equivalently}\qquad
 F_{n_{k+1}}\subseteq\interior(F_{n_k}).
\]
Thus every closed-tail representation of $\cK_\tau(X)$ has a regular
cofinal subsequence.
\end{cor}

\begin{proof}
Having chosen $n_k$, compactness of $\cl(U_{n_k})$ and the increasing
open cover $\langle U_n\rangle$ give $m$ such that
$\cl(U_{n_k})\subseteq U_m$.  Choose
$n_{k+1}>\max\{n_k,m\}$ and take complements.
\end{proof}

Thus the sequences representing $\cK_\tau(X)$ are precisely those whose
closed tails are the complements of a cofinal relatively compact open
exhaustion.  After passage to a subsequence, every such exhaustion has the
familiar regular nesting property.

\section{A counterexample to a 25-year-old conjecture}

Until now the topology of $X$ has remained fixed.  The final question of
\cite{Beer2001} changes the rules: the topology may be replaced by a metrizable supertopology in an attempt
to make the induced bornology metric.  Since successful supertopologies exist in
many natural situations, it is tempting to expect a least one.  The ordinary
sequence of positive integers in the real line shows that this expectation is
false.

Notice that every supertopology of $\tau$ preserves $\tau$-convergence to the empty set:
the terms remain closed and the closures of their tails can only decrease.
The final question of \cite{Beer2001} asks whether, for every metrizable
$(X,\tau)$ and every sequence of nonempty closed sets that is $\tau$-convergent to the empty set,
there is a coarsest metrizable supertopology $\sigma$ of $\tau$ for which
$\cB_\sigma(\langle T_n\rangle)$ is a nontrivial metric bornology.  We now
give a negative answer.

\begin{thm}\label{thm:counterexample}
Let $X=\mathbb R$ with its usual topology $\tau$, and let
$T_n=\{n\}$ for $n\in\N$.  There is no coarsest metrizable supertopology
$\sigma$ of $\tau$ for which
$\cB_\sigma(\langle T_n\rangle)$ is a nontrivial metric bornology.
\end{thm}

\begin{proof}
The sequence $\langle T_n\rangle$ is $\tau$-convergent to the empty set.

We divide the argument into three steps.

\smallskip
\emph{Step 1: the usual topology is not successful.}
We show that $\cB_\tau(\langle T_n\rangle)$ has no countable base.  Suppose
to the contrary that $\{B_j:j\in\N\}$ is such a base.  Since
$\cl_{\tau}(B_j)\cap\N$ is finite, choose inductively
\[
 1<n_1<n_2<\cdots,
 \qquad n_j\notin\cl_{\tau}(B_j).
\]
Choose $0<\varepsilon_j<\frac14$ so that
$(n_j-\varepsilon_j,n_j+\varepsilon_j)\cap B_j=\emptyset$, and set
\[
 x_j=n_j+\frac{\varepsilon_j}{2},
 \qquad A=\{x_j:j\in\N\}.
\]
The set $A$ is closed in $\mathbb R$ and misses $\N$, so it belongs
to $\cB_\tau(\langle T_n\rangle)$.  Yet $x_j\notin B_j$ for every
$j$, and hence $A$ is contained in no member of the alleged base.  By
Theorem~\ref{thm:Hu}, the bornology is not a metric bornology.

\smallskip
\emph{Step 2: a decreasing family of successful supertopologies.}
For $m\in\N$, put
\[
 F_m=\{m,m+1,m+2,\ldots\}
 \]
and let $\tau_m$ be the topology generated by $\tau\cup\{F_m\}$.  The
set $F_m$ is clopen in $\tau_m$, as is its complement
$Y_m=\mathbb R\setminus F_m$.  The subspace $Y_m$ carries its usual
topology, while $F_m$ is discrete.  Thus $\tau_m$ is metrizable.

For completeness, define a bounded compatible metric $\rho_m$ by
\[
 \rho_m(x,y)=
 \begin{cases}
  \min\{|x-y|,\frac12\},&x,y\in Y_m,\\
  0,&x=y,\\
  \frac12,&x\neq y\text{ and }x,y\in F_m,\\
  1,&\text{one of }x,y\text{ lies in }F_m\text{ and the other in }Y_m.
 \end{cases}
\]
The function
\[
 h_m(x)=
 \begin{cases}
  0,&x\in Y_m,\\
  x,&x\in F_m
 \end{cases}
\]
is $\tau_m$-continuous.  Therefore
\[
 d_m(x,y)=\rho_m(x,y)+|h_m(x)-h_m(y)|
\]
is an unbounded metric compatible with $\tau_m$.  Indeed, compatibility
follows from $d_m\geq\rho_m$ and the $\rho_m$-continuity of $h_m$:
$d_m$-convergence implies $\rho_m$-convergence, while
$\rho_m$-convergence and continuity of $h_m$ imply $d_m$-convergence.
Since $\rho_m$ is bounded, a set is $d_m$-bounded exactly when its
image under $h_m$ is bounded, which occurs exactly when its intersection
with $F_m$ is finite.
Since every $n\geq m$ is isolated in $\tau_m$,
\begin{equation}
 \cB_{\tau_m}(\langle T_n\rangle)
 =\{A\subseteq\mathbb R:A\cap F_m\text{ is finite}\}
 =\cB_{d_m}(\mathbb R).
 \label{eq:successful-supertopology}
\end{equation}
Thus every $\tau_m$ is a successful supertopology.  These supertopologies decrease with $m$: the set
\[
 F_{m+1}=F_m\cap(m+\tfrac12,\infty)
\]
is $\tau_m$-open, and hence
$\tau_{m+1}\subseteq\tau_m$.

\smallskip
\emph{Step 3: their intersection.}
We claim that
\begin{equation}
 \bigcap_{m=1}^{\infty}\tau_m=\tau.
 \label{eq:intersection-topologies}
\end{equation}
Only the reverse inclusion requires proof.  Let $U$ belong to every
$\tau_m$, and take $x\in U$.  Choose $m>x$.  Then $x\in Y_m$, and
the topology induced by $\tau_m$ on $Y_m$ is its usual subspace topology.
Since $U$ is $\tau_m$-open, $x$ has a usual open neighborhood
contained in $U$.  Hence $U\in\tau$, proving
\eqref{eq:intersection-topologies}.

Equation \eqref{eq:intersection-topologies} says that the original topology is
recovered exactly as the intersection of $\{\tau_m: m \in \mathbb{N}\}$.  If a
coarsest successful supertopology $\sigma$ of $\tau$ existed, then
$\sigma\subseteq\tau_m$ for every $m$.  Consequently
\[
 \tau\subseteq\sigma
 \subseteq\bigcap_m\tau_m=\tau,
\]
so $\sigma=\tau$.  But Step 1 already establishes that the usual topology is not successful, so we obtain a contradiction.
\end{proof}

\section{A general dichotomous criterion}

We conclude this article by presenting a general theorem for which the preceding example is a particular instance. It constitutes a true dichotomy for
the topologies: a smallest successful metrizable supertopology can occur
only when no proper supertopology is needed.

\begin{thm} \label{thm:coarsest-supertopology}
Let $(X,\tau)$ be a noncompact metrizable space, let
$\langle T_n\rangle$ be a sequence of nonempty closed sets converging to the empty
set, and set
\begin{equation}
 F_n=\bigcup_{j=n}^{\infty}T_j.
 \label{eq:coarsest-tails}
\end{equation}
The following conditions are equivalent.
\begin{enumerate}
\item Among the metrizable supertopologies of $\tau$, there is a coarsest
      supertopology $\tau_0$ for which $\cB_{\tau_0}(\langle T_n\rangle)$ is a
      nontrivial metric bornology.
\item The bornology $\cB_\tau(\langle T_n\rangle)$ is already a
      nontrivial metric bornology.
\item For every $k$ there exists $n$ such that
      \begin{equation}
       F_n\subseteq\interior_{\tau}(F_k).
       \label{eq:coarsest-regular}
      \end{equation}
\end{enumerate}
In condition (3), one may equivalently require $n>k$.  When these
conditions hold, the coarsest topology is $\tau$ itself.
\end{thm}

\begin{proof}
Since $\langle T_n\rangle$ converges to the empty set, the family of its terms is locally
finite.  Hence every $F_n$ is closed, and the sequence $\langle F_n\rangle$ converges to the empty set.  As observed in the proof of
Corollary~\ref{cor:regular-polar}, the bornology
$\cB_\tau(\langle T_n\rangle)$ is metric precisely when the closed tail
tower has a cofinal subsequence $\langle F_{n_j}\rangle$ satisfying
\[
 F_{n_{j+1}}\subseteq\interior_{\tau}(F_{n_j})
 \qquad(j\in\N).
\]
This is equivalent to \emph{(3)}.  Indeed, a regular cofinal subsequence
gives \eqref{eq:coarsest-regular} for every $k$, while
\eqref{eq:coarsest-regular} allows one to choose such a subsequence
recursively.  Moreover, if an index $n\leq k$ witnesses
\eqref{eq:coarsest-regular}, then
\[
 F_k\subseteq F_n\subseteq\interior_{\tau}(F_k)\subseteq F_k,
\]
so $F_k$ is open and $F_{k+1}\subseteq\interior_{\tau}(F_k)$.  Thus the
witness may always be chosen larger than $k$.  This proves the equivalence of
\emph{(2)} and \emph{(3)}.

If \emph{(2)} holds, then $\tau$ itself is successful and is contained in
every supertopology of $\tau$.  Hence \emph{(1)} holds, with
$\tau_0=\tau$.

Conversely, suppose that \emph{(1)} holds, and let $\tau_0$ be the
coarsest successful supertopology.  For each $m\in\N$, let $\sigma_m$ be the
topology generated by
\[
 \tau\cup\{F_n:n\geq m\}.
\]
Each $F_n$, $n\geq m$, is clopen in $\sigma_m$.  The topology
$\sigma_m$ is metrizable.  To see this, fix a bounded metric $d$ compatible
with $\tau$.  For each $n\geq m$, let $\mathbf 1_{F_n}$ denote the
characteristic function of $F_n$, equal to $1$ on $F_n$ and to $0$ on
$X\setminus F_n$.  Then the metric
\[
 \delta_m(x,y)
 =d(x,y)+\sum_{n=m}^{\infty}2^{-n}
   \bigl|\mathbf 1_{F_n}(x)-\mathbf 1_{F_n}(y)\bigr|
\]
induces exactly the topology generated by $\tau$ and the sets
$F_n$, $n\geq m$.

The family
\[
 \{X\setminus F_n:n\geq m\}
\]
is a countable clopen base for
$\cB_{\sigma_m}(\langle T_n\rangle)$.  Indeed, membership in this
bornology means that the $\sigma_m$-closure misses some tail $F_n$, and
the witnessing index may be taken at least $m$.  Conversely,
$X\setminus F_n$ is $\sigma_m$-closed and misses every $T_j$ with
$j\geq n$.  Theorem~\ref{thm:Hu} therefore shows that
$\cB_{\sigma_m}(\langle T_n\rangle)$ is a metric bornology, and therefore every $\sigma_m$ is successful.

The topologies decrease with $m$, and
\begin{equation}
 \bigcap_{m=1}^{\infty}\sigma_m=\tau.
 \label{eq:universal-intersection}
\end{equation}
To prove the reverse inclusion, let $U$ belong to every $\sigma_m$ and fix
$x\in U$.  Choose $m$ with $x\notin F_m$.  Since
$F_n\subseteq F_m$ for every $n\geq m$, the topology induced by
$\sigma_m$ on $X\setminus F_m$ is just the topology induced by $\tau$.
The set $X\setminus F_m$ is $\tau$-open, so $x$ has a $\tau$-open
neighborhood contained in $U$.  Hence $U\in\tau$, proving
\eqref{eq:universal-intersection}.

Since $\tau_0$ is coarsest, it is contained in every $\sigma_m$.  Therefore
\[
 \tau\subseteq\tau_0
 \subseteq\bigcap_m\sigma_m=\tau,
\]
and $\tau_0=\tau$.  In particular, $\tau$ is successful, which finally proves \emph{(2)}.
\end{proof}

The singleton sequence has now played both of its roles.  In the usual
topology it gives the simplest warning that an induced bornology need not be
metric.  It also makes the polar criterion visible.  Its closed tails
\[
 F_n=\{n,n+1,n+2,\ldots\}
\]
have empty interior in the usual topology, so the regular-tail condition of
Corollary~\ref{cor:regular-polar} fails.  In $\tau_m$, all tails $F_n$,
$n\geq m$, are clopen, and the closed tail tower becomes regularly nested.  Thus
the successful supertopologies do exactly what the polar theory predicts: they
regularize the closed sets at infinity.

The obstruction is therefore not a shortage of successful supertopologies.  The
topologies $\tau_m$ are explicit, metrizable, and decrease all the way to
the usual topology.  What fails is attainment: at the intersection the
induced bornology loses the countable base required by Theorem~\ref{thm:Hu}.  There
is no least successful supertopology.


\begin{thebibliography}{99}

\bibitem{Atsuji1958}
M.~Atsuji,
\emph{Uniform continuity of continuous functions of metric spaces},
Pacific J. Math. \textbf{8} (1958), 11--16; erratum, ibid., 941.

\bibitem{Attouch1984}
H.~Attouch,
\emph{Variational Convergence for Functions and Operators},
Pitman, Boston, 1984.

\bibitem{Beer1993}
G.~Beer,
\emph{Topologies on Closed and Closed Convex Sets},
Kluwer Academic Publishers, Dordrecht, 1993.

\bibitem{Beer2001}
G.~Beer,
\emph{Metric bornologies and Kuratowski--Painlev\'e convergence to the empty set},
J. Convex Anal. \textbf{8} (2001), 279--289.

\bibitem{Beer2023}
G.~Beer,
\emph{Bornologies and Lipschitz Analysis},
CRC Press, Boca Raton, 2023.

\bibitem{BeerPajoohesh2026}
G.~Beer and H.~Pajoohesh,
\emph{On a theorem of S.-T. Hu},
J. Convex Anal. \textbf{33} (2026), 615--628.

\bibitem{Carlson1989}
J.~W.~Carlson,
\emph{Balanced filters},
Topology Proc. \textbf{14} (1989), 25--40.

\bibitem{Engelking1977}
R.~Engelking,
\emph{General Topology},
Polish Scientific Publishers, Warsaw, 1977.

\bibitem{EshedEtAl2020}
A.~Eshed, M.~V.~Ferrer, S.~Hern\'andez, P.~Szewczak, and B.~Tsaban,
\emph{A classification of the cofinal structures of precompacta},
Ann. Pure Appl. Logic \textbf{171} (2020), 102810.

\bibitem{GarridoMerono2014}
M.~I.~Garrido and A.~S.~Mero\~no,
\emph{New types of completeness in metric spaces},
Ann. Acad. Sci. Fenn. Math. \textbf{39} (2014), 733--758.

\bibitem{Hejcman1959}
J.~Hejcman,
\emph{Boundedness in uniform spaces and topological groups},
Czechoslovak Math. J. \textbf{9} (1959), 544--563.

\bibitem{Hu1949}
S.-T.~Hu,
\emph{Boundedness in a topological space},
J. Math. Pures Appl. (9) \textbf{28} (1949), 287--320.

\bibitem{KleinThompson1984}
E.~Klein and A.~C.~Thompson,
\emph{Theory of Correspondences},
Wiley-Interscience, New York, 1984.

\bibitem{KunduAggarwal2023}
S.~Kundu and M.~Aggarwal,
\emph{Metric Spaces and Related Analysis},
World Scientific, Singapore, 2023.


\bibitem{Rainwater1959}
J.~Rainwater,
\emph{Spaces whose finest uniformity is metric},
Pacific J. Math. \textbf{9} (1959), 567--570.

\bibitem{Vaughan1937}
H.~E.~Vaughan,
\emph{On locally compact metrizable spaces},
Bull. Amer. Math. Soc. \textbf{43} (1937), 532--535.

\end{thebibliography}
\end{document}